\documentclass[10pt]{svjour3}  

\usepackage{url,color}
\usepackage{subfigure}
\usepackage{amsfonts,mathrsfs}
\usepackage{amssymb,amsmath}
\usepackage{verbatim}
\usepackage{acronym}
\usepackage{graphicx}

\newcommand{\dist}{{\rm dist}}

\providecommand{\rset}[1]{\mathbb{R}^}

\providecommand{\norm}[1]{\lVert#1\rVert}
\DeclareMathOperator{\ri}{ri}

\makeindex 

\begin{document}

\title{Linear convergence of dual coordinate descent on non-polyhedral convex problems
}

\titlerunning{Linear convergence of dual coordinate descent}        

\authorrunning{I. Necoara \textit{and} O. Fercoq }

\author{Ion Necoara$^{\dagger\ast}$ \and Olivier Fercoq $^{\ddagger}$}

\institute{$^{\dagger}$Ion Necoara \at Automatic Control and Systems Engineering Department, \\ 
University Politehnica Bucharest, \\ 
060042 Bucharest, Romania,\\ 
Email: \url{ ion.necoara@acse.pub.ro}
\and
$^{\ddagger}$Olivier Fercoq  \at LTCI, CNRS, Telecom ParisTech, \\
Universite Paris-Saclay, \\ 
75013 Paris, France, \\ 
Email: \url{ olivier.fercoq@telecom-paristech.fr. }
\and
$^{\ast}$\textit{Corresponding author} (\email{\tt ion.necoara@acse.pub.ro}). }

\date{Received: date / Accepted: date}

\maketitle

\begin{abstract}
This paper deals with constrained convex  problems, where the objective function is smooth  strongly convex and the  feasible set is  given as the intersection of a large number of closed convex (possibly non-polyhedral) sets.  In order to deal efficiently with the complicated constraints we consider a dual formulation of this problem.  We prove that the corresponding dual function satisfies a quadratic growth property  on any sublevel set, provided that the objective function is smooth and strongly convex and the sets verify the Slater's condition. To the best of our knowledge, this work is the first deriving a quadratic growth condition  for the dual under these general assumptions. Existing works derive similar quadratic growth conditions under more conservative assumptions, e.g.,    the sets need to be  either polyhedral or compact.  Then, for finding the minimum  of the dual problem, due to its  special composite structure,  we  propose  random (accelerated) coordinate descent  algorithms.  However, with the existing theory one can prove that such methods converge only  sublinearly. Based on our new quadratic growth property  derived for the dual,  we now show that such methods have faster convergence, that is the dual random (accelerated) coordinate descent  algorithms converge linearly.    Besides providing a general dual framework for the analysis of randomized coordinate descent  schemes, our results resolve an open problem in the literature related to the convergence of Dykstra algorithm on  the best feasibility problem for a collection of convex sets.  That is,   we  establish linear convergence  rate for the  randomized Dykstra algorithm when the convex sets satisfy the Slater's condition and derive also a new accelerated variant for the Dykstra algorithm. 


\end{abstract}

	
\section{Introduction}
\noindent   The main problem of interest in this paper is the minimization of a smooth strongly convex function over the intersection of a finite number of convex (possibly non-polyhedral) sets:
\begin{align}
\label{bap-vv}
 \min\limits_{x \in \rset^n} & \; g(x)    \quad   \text{s.t.}  \;\;\; x \in X:=\bigcap_{i=1}^m  X_i,
\end{align}
where we assume that the objective function $g : \rset^n \to \rset^{} $ is smooth and strongly convex. Moreover, we consider that the number of sets  $m$ from the intersection is very large and each set $X_i$ is closed convex and simple (by simple we mean that one can easily project onto that set, e.g., hyperplanes, halfspaces, balls, etc). This model covers, in particular,  feasibility problems (see \cite{BauBor:96} for a survey), such as the best approximation problem that consists of  finding the projection of a given point $v \in \rset^n$ in the intersection of some closed convex sets \cite{BoyDyk:86}: 
\begin{align}
\label{bap}
\min\limits_{x \in \rset^n} & \;  \frac{1}{2}\norm{x -v}^2    \quad   \text{s.t.}  \;\;\; x \in \bigcap_{i=1}^m  X_i.
\end{align} 
Note that the linear support vector machine (SVM)  can be  formulated as problem \eqref{bap}, where $v=0$ and each set $X_i$ is a given halfspace.  However, the domain of applicability of  optimization model  \eqref{bap-vv} extends beyond feasibility problems and SVM. For example, when applying an (accelerated) gradient or proximal point algorithm  to solve the convex optimization problem (that covers, in particular, the large class of cone programming) 
$$ \min_{x \in \rset^n} \; \phi(x) \;\;\;  \text{s.t.} \;\;\;  x \in \bigcap_{i=1}^m  X_i, $$  
we need in each iteration to find an approximate solution of the following subproblem  for a given point $\tilde x$ and a parameter $\alpha >0$ \cite{Nes:04}: 
\[   \min_{x \in \rset^n} \;  \ell(x; \tilde x) + \frac{1}{2\alpha} \| x - \tilde x\|^2 \quad  \text{s.t.}  \quad x \in \bigcap_{i=1}^m  X_i,  \]
where either $ \ell(x; \tilde x) = \phi(\tilde x) + \langle  \nabla \phi(\tilde x), x - \tilde x \rangle$, when (accelerated) gradient algorithm is applied, or  $ \ell(x; \tilde x) = \phi(x) $, when (accelerated) proximal point  algorithm is used, respectively.  Clearly, this subproblem fits into the settings considered  for problem  \eqref{bap-vv}, since in this case the objective function $g(x) = \ell(x; \tilde x) + \frac{1}{2\alpha} \| x - \tilde x\|^2$   is always strongly convex and also smooth provided that e.g. $\phi$ is smooth.  Optimization problem  \eqref{bap-vv}  can be also used as a  modeling paradigm for solving many engineering problems such as   radiation therapy treatment planning \cite{HerChe:08}, magnetic resonance imaging \cite{SamKho:04}, wavelet-based denoising \cite{ChoBar:04}, color imaging \cite{Sha:00}, antenna design \cite{GuSta:04}, sensor networks \cite{BlaHer:06}, data compression \cite{ComPes:11,LieYan:05}, neural networks \cite{StaYan:98}  and  optimal control \cite{PatNec:17}.  

\vspace{0.2cm}

\noindent Our goal is to devise efficient algorithms with mathematical gurantess of convergence for solving the optimization problem \eqref{bap-vv}, and, in particular, \eqref{bap},  when $m$ is large and the sets $X_i$'s are not all polyhedral.   Basically, we can identify three popular classes of algorithms to solve such optimization problems: interior-point, active set and first order methods \cite{Nes:04}. However, the first two classes of algorithms encounter numerical difficulties when $m$ is large.  Furthermore, although primal projected  first order algorithms (including coordinate descent type schemes) achieve linear convergence for smooth strongly convex constrained minimization, they require exact projection onto the feasible set $X$ \cite{Nes:04}. Note that  the projection problem \eqref{bap} can be as difficult as the original problem \eqref{bap-vv}, hence, when the projection onto the feasible set $X$ is complicated  primal first order methods are also not applicable.\\    

\noindent Algorithmic alternatives to convex problems with complicated feasible set are the dual first order methods.  Note that one of the most efficient  projection schemes for  the best approximation problem \eqref{bap} is the Dykstra algorithm \cite{BoyDyk:86,ComPes:11}, which can be interpreted as a dual coordinate descent scheme. Dual gradient-based  methods are able to handle easily complicated constraints, but they have typically sublinear convergence rate  even when the primal problem has smooth and strongly convex objective function \cite{NecNed:14}.   There are few exceptions: e.g., \cite{NecNed:15}  proves that the dual function of   \eqref{bap-vv}  satisfies an error bound condition, provided that $g$ is smooth and strongly convex and $X_i$'s are polyhedral sets;  \cite{GidPed:18} proves that augmented dual function  of   \eqref{bap-vv}  satisfies a quadratic growth condition, provided that $g$ is smooth and strongly convex and $X_i$'s are bounded sets.  Error bound and quadratic growth conditions  are equivalent and both represent relaxations of the strong convexity condition of a function,  see \cite{NecNes:15} for more details.  Under these relaxation conditions  one can  prove linear convergence for first order  methods (including coordinate descent type algorithms) \cite{NecNes:15,NecCli:16,FerQu:18}.  However, requiring $X_i$'s to be all either polyhedral or bounded sets, restricts drastically the domain of applicability of the optimization problem \eqref{bap-vv} (e.g., we cannot tackle second-order cone programming).  In this paper we prove that the dual function of  \eqref{bap-vv}, with $g$ smooth and strongly convex and the sets $X_i$'s verifying Slater's condition,  satisfies a quadratic growth property  on any sublevel set.  Then, we show that coordinate descent-based methods are converging linearly when solving the dual problem.  More precisely, the main contributions of this paper are:

\begin{enumerate}
\item[$(i)$] We first derive a composite dual formulation of the  convex problem  \eqref{bap-vv} that is formed as a sum of two convex terms: one is smooth and another is general but simple and separable. Then, we prove a   quadratic growth property on any sublevel set  for this dual  composite function  under  the assumptions that the objective function $g$ is smooth and strongly convex and the general (possibly non-polyhedral) convex sets $X_i$'s satisfy the Slater's condition. Hence, our result extends in a nontrivial way the existing results of  \cite{NecNed:15} (for polyhedral sets) and  \cite{GidPed:18} (for bounded convex sets). 

\item[$(ii)$] Given the structured composite form for  the dual problem  of  \eqref{bap-vv} we consider dual (accelerated) coordinate descent algorithms for solving it.   It is well-known that such methods  converge sublinearly when solving  smooth (dual) convex problems \cite{Nes:12,FerRic:15,LuXia:14,RicTak:14}, although in practice one can observe a faster (linear) convergence.  However, based on the quadratic growth property derived in this paper for the dual, we now prove that  dual (accelerated) coordinate descent algorithms have faster convergence rates, that is they converge linearly.   

\item[$(iii)$]  As a consequence of our results,  we implicitly establish linear rate of the classic  randomized Dykstra algorithm for solving the dual  of the best feasibility problem \eqref{bap}.   From our knowledge, Dykstra algorithm was proved to converge linearly only for polyhedral sets  and it is a long standing open question whether a similar result holds for more general sets \cite{BoyDyk:86,ComPes:11,DeuHun:94}. We  answer positively to this open question, proving that randomized Dykstra algorithm is converging linearly for the general class of sets satisfying the Slater  condition. We also derive for the first time an acceleration of  the Dykstra algorithm, which also converge linearly and usually faster than basic Dykstra. 
\end{enumerate}

\noindent Let us emphasize the following points of our contributions. Firstly, although our proof  for the quadratic growth property  uses some ideas from \cite{GidPed:18}, it requires new concepts and techniques,  since we are dealing with (possibly) unbounded and non-polyhedral sets verifying just Slater's condition.  Second, using the quadratic growth we can prove faster convergence rates  for dual (accelerated) coordinate descent algorithms than was previously known.  Thirdly, the  Dykstra algorithm was known to converge linearly only for polyhedral sets \cite{DeuHun:94,Pan:17}. Since Dykstra can be interpreted as a  coordinate descent scheme for solving the dual of \eqref{bap}, we now show  that it converges linearly on general sets satisfying Slater's condition.   We also derive an accelerated variant of Dykstra algorithm, which, usually, has better convergence rate than its non-accelerated counterpart.     

\vspace{0.2cm}
	
\noindent \textit{Notation}. We denote by $\Pi_{X} (x)$ the projection of the point $x$ onto the convex set $X$. We also denote $\text{dist}(x,X) = \min_{z \in X}\norm{z-x} = \| x - \Pi_{X} (x) \|$. The relative interior of $X$ is denoted by $\ri(X)$. For a convex function $g : \rset^n \to \rset^{}$ we define its Fenchel conjugate as $g^*(y)  =  \max_{x \in \rset^n} \langle x,y\rangle -g(x)$.  The indicator function of $X$ is denoted by $\mathbb{I}_X(\cdot)$.  For the support function of the convex set $X$, which is the Fenchel conjugate of the indicator function,  we use the notation $\text{supp}_X(y) := \max_{x \in X} \langle y, x\rangle$.  For a given $x \in X$ we denote the normal cone by ${\cal N}_{X}(x) = \{y: \; \langle y, z-x \rangle \leq 0 \;\; \forall z \in X   \}$. For simplicity we omit the transpose, i.e. instead of  $y=(y_1^T \cdots y_M^T)^T$ we use  $y=(y_1;\cdots;y_M)$, where $y_i \in \rset^n$.


\section{Preliminaries}
We aim at minimizing a smooth strongly convex  function over the intersection of a large  number of simple closed convex sets \eqref{bap-vv}, which for convenience we recall it again here:
\begin{align}
\label{bap-v}
\min\limits_{x \in \rset^n} & \; g(x)    \quad   \text{s.t.}  \;\;\; x \in \bigcap_{i=1}^m  X_i.
\end{align}  
We recall  that $g: \rset^n \to \rset^{}$ is smooth (i.e., it has Lipschitz continuous gradient) and   strongly convex if there exist constants $0 < \sigma < L$ such that the following inequalities hold:
\[   \frac{\sigma}{2} \|y - x\|^2 \leq g(y) - g(x) - \langle \nabla g(x), y-x \rangle \leq \frac{L}{2} \|y - x\|^2 \quad \forall x, y \in \rset^n. \]    
Note that by a proper scalling of $g$ we can always assume $\sigma=1$. Therefore, in the sequel we consider $1$-strongly convex function $g$ and with $L$-Lipschitz continuous gradient.  We denote the intersection by $X : = \bigcap_{i=1}^m X_i$ and we consider that the projection onto each sets $X_i$ can be computed efficiently. For example, when the set $X_i$ is  a hyperplane, a halfspace or a ball the projection can be computed in   closed form.  Further,  let us define a few fundamental properties  on the sets $X_i$, which are often considered in the literature, see e.g.,  \cite{BauBor:96,BecTeb:03,Ned:11,NecRic:18}. For simplicity, we consider a uniform probability distribution over the set $[m]$ and thus for any scalar random variable $\theta_i$ we define its expectation as $ \mathbb{E}[  \theta_i] = 1/m \sum_{i \in [m]} \theta_i$.
	
\begin{definition}[(Bounded) Linear Regularity]
\label{def:1}
The collection of sets	$\{X_i\}_{i=1}^{m}$ has bounded linear regularity property if  for any $r>0$ there exists $\mu > 0$ such that:
\begin{align}
\label{blr}
\mu \cdot \text{dist}^2(x,X)   \le \mathbb{E}[\text{dist}^2(x,X_i)]  \qquad \forall x  \in B(0;r).
\end{align}
When the previous inequality holds for any $x \in \rset^n$ we say that the sets have linear regularity property:
\begin{align}
\label{lr}
\mu \cdot \text{dist}^2(x,X)  \le \mathbb{E}[\text{dist}^2(x,X_i)]  \qquad \forall x  \in \rset^n.
\end{align}
\end{definition}
	
\noindent It follows  from Definition  \ref{def:1} that $\mu \in (0, \; 1]$. Indeed, since   $\text{dist}(x,X_i) \leq   \text{dist}(x,X)$ for all $i \in [m]$, we have: 
\[ \mu \cdot \text{dist}^2(x,X)   \le \mathbb{E}[\text{dist}^2(x,X_i)]  \leq  \mathbb{E}[\text{dist}^2(x,X)] =  \text{dist}^2(x,X),   \]
which proves that $\mu \in (0, \; 1]$. Note that $\mu$ is related to the condition number of the set intersection $X = \bigcap_{i=1}^m  X_i$, see  \cite{NecRic:18} for a detailed discussion.  Moreover, $\mu=1$ is the ideal case, while $\mu$ close to zero is the difficult case (in fact, for $\mu=0$ inequality \eqref{lr} always holds).  Let us also recall  the Slater's condition:	

\begin{definition}[Slater's condition]
Let	$\{X_i\}_{i=1}^m$ be a collection of closed convex sets. Also assume that there is $0 \leq r \le m$ such that the sets $X_{r+1}, \cdots, X_m $ are polyhedral. Then, the sets  $\{X_i\}_{i=1}^m$ satisfy \textit{Slater's condition} if the following property holds:
\begin{align}
\label{slater}
\left(\bigcap_{i=1}^r \ri(X_i) \right) \bigcap \left( \bigcap_{i=r+1}^m X_i \right) \neq \emptyset.
\end{align}
\end{definition}

\noindent We now recall the following classical result stating the relation between the two definitions given above, whose  proof can be found e.g.,  in \cite{BauBor:99} (Corrolary 3 and 6).
\begin{corollary}
\label{th:slater-lr}
Let $\{X_i\}_{i=1}^{m}$ be a finite collection of closed convex sets, where for some $r \in [0:m]$, the sets $X_{r+1}, \cdots, X_m $ are polyhedral. Suppose that the Slater's condition \eqref{slater} holds. Then, the bounded linear regularity  \eqref{blr} holds.  Moreover, if $X$ is bounded, then the global linear regularity property \eqref{lr} holds. 
\end{corollary}

\begin{remark}
Note that  the linear regularity condition can be weaker than Slater's condition. For example, let us consider   the intersection of the epigraph of the function  $f_1(x)=\sqrt{x^2+x^4}$ and of the hypograph of the function $f_2(x)=-\sqrt{x^2+x^4}$. Then, it is easy to check that linear regularity holds, while Slater fails. 
\end{remark}	

\noindent Another important notion in optimization is the quadratic growth property of an objective function of an optimization problem, see \cite{NecNes:15} for a detailed exposition.  	
\begin{definition}[(Local) Quadratic Growth]
\label{def:1}
The  convex function $d: Y \to \rset^{}$, with $Y \subseteq \rset^m$, has the  local quadratic growth property  if  for any $y^0  \in \rset^m$ there exists $\sigma > 0$ such that:
\begin{align}
\label{lqg}
d(y) - d^* \geq \frac{\sigma}{2} \text{dist}^2(y,Y^*) \qquad \forall y  \in Y \cap \{y: d(y) \leq d(y^0) \},
\end{align}
where $d^*$ and $Y^*$ are the optimal value and the optimal set of $\min_{y \in Y} d(y)$, respectively.  When the previous inequality holds for any $y \in Y$ we say that the function $d$ has the quadratic growth property. 
\end{definition}

\noindent It is known that the class of functions satisfying the (local) quadratic growth is larger than the class of strongly convex functions. For example, any function of the form $d(y) = D(A^Ty)$, where $D$ is a strongly convex function and $A \not =0$ is any matrix of appropriate dimension, satisfies  the quadratic growth property \cite{NecNes:15}.  In particular, the dual of a primal problem with linear constraints, $\min_{x: Ax \leq b} g(x)$, satisfies the quadratic growth condition \eqref{lqg}, provided that $g$ is a smooth strongly convex function, see \cite{NecNed:14}.    Moreover, it has been shown recently that gradient-based algorithms, such as (projected) gradient method, restarted accelerated gradient method and their random coordinate descent counterparts   converge linearly on the class of  convex problems whose objective function is smooth and  satisfies the local quadratic growth property \cite{NecNes:15,NecCli:16,FerQu:18}.  These existing results motivate us to investigate further the properties of the dual function corresponding to the more general primal problem \eqref{bap-v}. In the next section we prove that  the dual function satisfies a  local quadratic growth property on any sublevel set, provided  that the objective function $g$ is smooth and strongly convex and the sets $X_i$'s satisfy the Slater's condition. 
	

\section{Dual formulation and properties}
\noindent In this section we take a close look at the primal convex problem \eqref{bap-v} and compute its dual form. Then, we analyze the main properties of the dual, in particular we prove a quadratic growth condition for the dual objective function.  Note  that the  convex optimization problem  \eqref{bap-v} can be equivalently written as:
\begin{align}
\label{model_indicator}
g^* = \min\limits_{x \in \rset^n} & \; g(x) + \sum\limits_{i=1}^m \mathbb{I}_{X_i}(x).
\end{align}
\noindent Further, by replicating the variable  $x$ in the model \eqref{model_indicator} we can obtain the following equivalent problem:
\begin{align}
\label{eq:extended}
\min\limits_{\textbf{x} \in \rset^{n(m+1)}} & \;  g(x) + \sum\limits_{i=1}^m \mathbb{I}_{X_i}(x_i) \\
\text{s.t.} & \;\; x = x_i \quad \forall i\in [m]. \nonumber
\end{align}
Since we want  to derive the  dual problem of \eqref{bap-v}, we form first  the Lagrangian function $\mathcal{L}: \rset^{n(m+1)} \times \rset^{mn} \mapsto \rset^{} \cup \{+\infty\}$ associated to the above problem and express the dual function as:
\begin{align*}
D(y) = \min\limits_{\textbf{x} \in \rset^{n(m+1)}} & \; \mathcal{L}(\textbf{x},y) \quad \left( : = g(x) + \sum\limits_{i=1}^m \mathbb{I}_{X_i}(x_i) + \sum\limits_{i=1}^m \langle y_i, x-x_i \rangle \right) \quad \forall y \in \rset^{mn},
\end{align*}	
where $\textbf{x} =(x; x_1;\cdots;x_m)$ and $y = (y_1;\cdots;y_m)$ (recall that for the simplicity of the notation we omit the transpose, i.e. instead of $y=(y_1^T \cdots y_M^T)^T$ we write $y = (y_1;\cdots;y_m)$).  Taking into account that the Fenchel  conjugate of the indicator function of a convex set $X$ is the support function, i.e.    $(\mathbb{I}_X)^*(\cdot) = \text{supp}_{X} (\cdot)$, where $\text{supp}_{X} (y) = \max_{x \in X} \langle x, y\rangle$, and denoting $d(y) = - D(y) $,  results into the following dual problem:
\begin{align}
\label{dbap}
d^* = \min\limits_{y \in \rset^{mn}} & \;  d(y) \;\;  \left( := \underbrace{ g^*\left(- \sum\limits_{i=1}^m y_i \right)}_{=\tilde{d}(y)}   + \underbrace{  \sum\limits_{i=1}^m \text{supp}_{X_i}(y_i)}_{=\text{supp}(y)} \right).
\end{align}
Recall that if $g:\rset^n \to \rset^{}$ has $L$-Lipschitz continuous  gradient, then its Fenchel conjugate $g^*$ is $1/L$-strongly convex function.  Similarly, if $g$ is $\sigma$-strongly convex, then its Fenchel conjugate $g^*$ has $1/\sigma$-Lipschitz continuous gradient, see e.g.,  \cite{RocWet:98}.


\subsection{  Zero duality gap for dual problem under linear regularity }
It is well-known that if  the  Slater's  condition \eqref{slater} holds for the collection of convex sets  $\{X_i\}_{i=1}^{m}$, then   the dual problem \eqref{dbap} is equivalent with the primal problem \eqref{bap-v}, i.e. strong duality holds \cite{RocWet:98}. More precisely, $g^*= - d^*$ and if we denote the optimal set of the dual problem  by $Y^* = \arg\min_y d(y)  \subseteq  \rset^{mn}$, then for any  dual  optimal solution $y^*=(y_1^*; \cdots; y_m^*) \in Y^* $ we can recover a primal optimal solution $x^*$ through the relation:
\begin{align}\label{pd-optcond}
  x^* = \nabla g^*\left(- \sum_{j=1}^M y_j^*\right).    
\end{align}

\noindent However, strong duality holds under the  more general linear regularity condition \eqref{lr}, as proved in the next theorem: 

\begin{theorem}
\label{th:sd}
For the  collection of sets  $\{X_i\}_{i=1}^{m}$ assume that their intersection is nonempty and  that they satisfy the  linear regularity condition \eqref{lr}. Then, strong duality holds, i.e. $d^* + g^* =0$. 
\end{theorem}

\begin{proof}
\noindent First, we observe that if  intersection is nonempty and  the linear regularity \eqref{lr} holds, then for any  $x \in X = \cap_{i=1}^m  X_i$ we have:
\[  {\cal N}_{ \cap_{i=1}^m  X_i}(x) = \textbf{+}_{i=1}^m   {\cal N}_{ X_i}(x), \]
where ${\cal N}_{X}(x)$ denotes the normal cone of the closed convex set $X$ at $x$. Hence,  $x$ is an optimal solution of the convex problem  \eqref{model_indicator} if and only if it satisfies:
\begin{align}
\label{eq:optcondpr}
0 \in \nabla g(x) + \left(\textbf{+}_{i=1}^m {\cal N}_{X_i}(x) \right). 
\end{align}
 Moreover, $y=(y_1;\cdots;y_m)$ is an optimal solution for the dual problem \eqref{dbap}  if and only if it satisfies:
\begin{align}
\label{eq:optconddu}
 0 \in -\nabla g^*\left(-\sum_{i=1}^m y_i\right) + {\cal N}_{X_i}^{-1}(y_i) \qquad \forall i. 
\end{align}
Now, if $x$ is an optimum for \eqref{model_indicator} , then from \eqref{eq:optcondpr} it follows that there are $y_i \in  {\cal N}_{X_i}(x)$ such that $\nabla g(x) = -\sum_{i=1}^m y_i$, or equivalently $x = \nabla g^*\left(-\sum_{i=1}^m y_i\right)$. Hence, for all $i$ we have $y_i \in  {\cal N}_{X_i}(\nabla g^*\left(-\sum_{i=1}^m y_i\right))$, or equivalently $\nabla g^*\left(-\sum_{i=1}^m y_i\right) \in {\cal N}_{X_i}^{-1}(y_i) $ for all $i$. Then,    $y=(y_1;\cdots;y_m)$ defined above satisfies \eqref{eq:optconddu} and thus optimal for the dual problem \eqref{dbap}.  Finally, let us note that there is no duality gap, since we have:
\begin{align*}
d^* & = g^*\left(-\sum_{i=1}^m y_i \right)  + \sum_{i} \text{supp}_{X_i}(y_i)  = -g(x) + \langle \nabla g(x), x\rangle  + \sum_{i=1}^m \langle y_i, x\rangle \\
& = - g(x) = - g^*.
\end{align*}  
\hfill$\square$
\end{proof}


\subsection{Regularity of dual function under Slater's condition}
 Recall that $Y^*$ denotes the set of dual solutions  of \eqref{dbap}.   In the sequel we  prove that the dual function satisfies a quadratic growth condition on any sublevel set, i.e., for any $y^0$ there exists $\sigma >0$ such that:
\[ d(y) - d^* \geq \frac{\sigma}{2} \dist^2(y, Y^*), \quad \forall y: d(y) \leq d(y^0). \]
To some extent, our proof of this fact is based on similar ideas as in \cite{GidPed:18}.  More precisely, instead of  our formulation \eqref{eq:extended}, the following problem is considered in \cite{GidPed:18}:
\[  \min_{\hat{\textbf{x}}=(x_1;\cdots;x_m)}  g(x_1,\cdots,x_m)  + \sum\limits_{i=1}^m \mathbb{I}_{X_i}(x_i) \quad  \text{s.t.}  \;\; M \hat{ \textbf{x}}=0, \]
where $g$ is assumed smooth and strongly convex, $M$ is any matrix, and $X_i$'s are all non-polyhedral compact sets satisfying Slater's condition.  Moreover, the assumption of bounded sets $X_i$ is crucial in the proofs of \cite{GidPed:18}. Under these settings,  \cite{GidPed:18} proves that the corresponding dual satisfies a quadratic growth condition on a ball around $Y^*$, i.e.,  for all $y$ satisfying  $\text{dist}(y,Y^*) \leq R$ for some appropriate $R$.  Our results are stronger, since we remove the assumption of bounded sets (thus we can also consider polyhedral sets in the intersection) and we prove that the quadratic growth holds on any sublevel  set instead of a ball around $Y^*$.  Note that  in our case the matrix $M$ of  \eqref{eq:extended} has the  form:
$$M =
\begin{bmatrix}
-I_n & I_n & 0 & \cdots & 0 \\
-I_n & 0 & I_n & \cdots & 0 \\
\cdots   & \cdots  & \cdots  & \cdots & \cdots \\
-I_n & 0 & 0 & \cdots & I_n \\
\end{bmatrix} \in \rset^{nm\times n(m+1)}$$ and the dual  function can be written as 
\begin{equation}
d(y) = g_{\textbf{X}}^*(-M^T y), \;\; \text{where} \;\;  g_{\textbf{X}}^*(z) = \sup_{\textbf{x}} \left(\langle z, \textbf{x} \rangle - g(x) -  \sum\limits_{i=1}^m \mathbb{I}_{X_i}(x_i)\right), 
\end{equation}
where $\textbf{X} = \rset^n \times X_1 \times \ldots \times X_m$ and recall that $\textbf{x} =(x; x_1;\cdots;x_m)$ and $y = (y_1;\cdots;y_m)$. First, we show the following lower bound on  $d(y) - d^*$:

\begin{lemma}
\label{lem:lower_bound_on_D}
Assume that $g$ is   $1$-strongly convex  function and with $L$-Lipschitz continuous gradient and $X_i$'s are general convex sets. Then, for any constant $R>0$ we have for all  $y$ satisfying  $\dist(y, Y^*) \leq R$ the inequality:  
\[d(y) - d^* \geq 
 \min\left(\frac{\sup_{u \in  \textbf{X} \cap \textbf{B}} -  \langle y - y^*, M u \rangle}{2}, \frac{(\sup_{u \in {\textbf{X}} \cap \textbf{B}} -  \langle y - y^*, M u \rangle)^2}{2 L D_R^2}\right), 
 \]
where $\textbf{B}$ denotes the ball with the center  $\textbf{x}^*=(x^*;\cdots;x^*) \in \textbf{X}$ and the radius $D_R = 2  \|M\| R/L$ and $x^*$ is the unique optimal solution of \eqref{bap-v}.
\end{lemma}

\begin{proof}
Let $ \textbf{x}  \in  \textbf{X} $ and define $g_{\textbf{X}} (\textbf{x}) = g(x) + \sum\limits_{i=1}^m \mathbb{I}_{X_i}(x_i)$.  Further, let us consider $t \in \partial g_{\textbf{X}} (\textbf{x}) = \nabla g(x) + {\cal N}_{\textbf{X}} (\textbf{x})$ and  define the function:
\[ \ell_{\textbf{x}}(u) = g_{\textbf{X}}(u+\textbf{x}) - g_{\textbf{X}}(\textbf{x}) - \langle u, t \rangle. \]
Note that  $\ell_{\textbf{x}}$ has a  minimum at $u=0$ with the optimal value $0$. Since $g$ is $L$-smooth, then we have:
\[ \ell_{\textbf{x}}(u) \leq h_{\textbf{x}}(u) \;\; \left(:= \frac{L}{2} \|u\|^2 + I_{\textbf{X}}(u+\textbf{x}) \right) \quad \forall u. \]
This implies that their Fenchel conjugates satisfy \cite{RocWet:98}: 
\[
\ell_{\textbf{x}}^*(v) \geq h_{\textbf{x}}^*(v) \quad  \forall v.
\]

\noindent Now, by noticing that  $\ell_{\textbf{x}}^*(v)   = g_{\textbf{X}} (v+t) - g_{\textbf{X}} (t) - \langle \textbf{x}, v \rangle$ and denoting $y^* = \Pi_{Y_*}(y)$, we have $-M^T y^* \in \partial g_{\textbf{X}} (\textbf{x}^*)$ and  $M\textbf{x}^*= 0$.   Hence, for any dual variable $y$ and for the unique optimal  $\textbf{x}^*$ we get:
\[
d(y) - d^* = g_{\textbf{X}}^*  (-M^T y) - g_{\textbf{X}}^*(-M^T y^*) = \ell_{\textbf{x}^*}^*(-M^\top(y - y^*)) \geq  h_{\textbf{x}^*}^* (-M^\top(y - y^*)).
\]
We also have:
\[
h_{\textbf{x}^*}^*(v) = \sup_{u} \left( \langle v, u \rangle - \frac L2 \|u\|^2 - I_{\textbf{X}}(u+\textbf{x}^*)\right) = \sup_{u \in \textbf{X}} \left(  -\frac L2 \|u - \textbf{x}^*\|^2 + \langle v, u - \textbf{x}^*\rangle \right).
\]
Hence, we obtain: 
\begin{align*}
d(y) - d^* & \geq \sup_{u \in \textbf{X}} \left( -\frac L2 \|u - \textbf{x}^*\|^2 + \langle -M^T (y - y^*), u - \textbf{x}^*\rangle \right) \\
& =  \sup_{u \in \textbf{X}}  \left( -\frac L2 \|u - \textbf{x}^*\|^2 - \langle y - y^*, M u \rangle \right).
\end{align*}
Note  that   $\sup_{u \in \textbf{X}} - \langle y - y^*, M u \rangle > 0$ for all $y \not \in Y^*$. Indeed, if $y$ is not a dual optimal point, then $-M^\top y \not \in \partial g_{\textbf{X}} (\textbf{x}^*)$. On the other hand,  $-M^T y^* \in \partial g_{\textbf{X}} (\textbf{x}^*)$ and  thus   there exists  $ u \in  \textbf{X}$ such that:
\begin{equation}
\label{eq:positive_for_any_y}
0 < \langle M^\top y^* - M^\top y , u - \textbf{x}^* \rangle = -\langle y - y^*, M u \rangle.
\end{equation}
As the linear term in  the expression $-\frac L2 \|u - \textbf{x}^*\|^2 - \langle y - y^*, M u \rangle$ can be made positive, choosing $u$ sufficiently close to $\textbf{x}^*$ shows that the  bound 
$$\sup_{u \in \textbf{X}}  \left( -\frac L2 \|u - \textbf{x}^*\|^2 - \langle y - y^*, M u \rangle \right)$$  is positive. 
Notice that since for any $u$ satisfying  $ -\frac L2 \|u - \textbf{x}^*\|^2 - \langle y - y^*, M u \rangle \geq 0$, we have: 
\begin{align*}
\frac L2 \|u - \textbf{x}^*\|^2 & \leq - \langle y - y^*, M u \rangle = - \langle M^\top (y - y^*), u - \textbf{x}^* \rangle\\
& \le \| M^\top (y - y^*)\|\; \|u - \textbf{x}^* \|.
\end{align*}
Using now the assumption on the distance from $y$ to $Y_*$, the last inequality implies $ \|u - \textbf{x}^*\| \leq \frac 2L \| M \| R$, which  shows that the supremum in the lower bound cannot be reached outside the ball  $\textbf{B}$ with the center  $\textbf{x}^*=(x^*;\cdots;x^*)$ and the radius $D_R = 2  \|M\| R/L$. Therefore, we get:
\begin{align*}
d(y) - d^* & \geq \sup_{u \in \textbf{X}}  \left( -\frac L2 \|u - \textbf{x}^*\|^2 - \langle y - y^*, M u \rangle \right) \\ 
& = \sup_{u \in \textbf{X} \cap \textbf{B}}  \left( -\frac L2 \|u - \textbf{x}^*\|^2 - \langle y - y^*, M u \rangle \right).
\end{align*}
Now, using the  change of variable $u' = (1-\gamma)x^* + \gamma u$, we further have: 
\begin{align*}
d(y) - d^* & \geq \sup_{u \in  \textbf{X} \cap \textbf{B}} \sup_{\gamma \in [0,1]}  \left( -\frac {\gamma^2L}{2} \|u - \textbf{x}^*\|^2 - \gamma \langle y - y^*, M u \rangle \right) \\ 
& \geq  \sup_{\gamma \in [0,1]} \sup_{u \in  \textbf{X} \cap \textbf{B}}  \left( -\frac {\gamma^2L}{2} D_R^2 - \gamma \langle y - y^*, M u \rangle \right) \\
& = \sup_{\gamma \in [0,1]} \left( -\frac {\gamma^2L}{2} D_R^2 + \gamma \sup_{u \in  \textbf{X} \cap \textbf{B}} - \langle y - y^*, M u \rangle  \right) \\
&= \min\left(\frac{\sup_{u \in  \textbf{X} \cap \textbf{B} }-  \langle y - y^*, M u \rangle}{2}, \frac{(\sup_{u \in  \textbf{X} \cap \textbf{B}}-  \langle y - y^*, M u \rangle)^2}{2LD_R^2}\right),
\end{align*}
which confirms our statement. \hfill$\square$
\end{proof}

\noindent Another  key result in our analysis is the following:

\begin{lemma}
\label{lem:structure_Y_star}
Assume that the sets $X_i$'s satisfy Slater's condition,  that is there exists $1 \leq r \le m$ such that the sets $X_{r+1}, \cdots, X_m $ are polyhedral and there exists $\bar x \in 
\left(\bigcap_{i=1}^r \ri(X_i) \right) \bigcap \left( \bigcap_{i=r+1}^m X_i \right) $. Moreover, assume that  $g$ is 1-strongly convex and with $L$-Lipschitz continuous gradient. Then, the dual optimal set $Y^*$ of \eqref{dbap} can be written as:
\[ Y^* = V + \mathcal K,  \]
where $V = \cap_{i=0}^m \big(M_{i,:}(\mathrm{Span}(X_i - \bar x))\big)^{\perp}$ and $\mathcal K$ is a compact set (we use $M_{i,:}$ to denote the appropriate block  column submatrix  of the matrix $M$ and recall that $X_0 = \rset^n$).
\end{lemma}

\begin{proof}
Let us decompose any dual optimum as $y^* = y^*_V + y^*_{V^\perp} \in Y^*$, where $y^*_V \in V$ and $y^*_{V^\perp} \in V^\perp = \oplus_{i=0}^m \big(M_{i,:}(\mathrm{Span}(X_i - \bar x))\big)$. We also denote $\bar{\textbf{x}} = (\bar x;\cdots;\bar x)$, the matrix $U=(I_n; \textbf{0})$  and $x^*$ the unique primal solution. Using the optimality condition $M^T y^* + U \nabla g(x^*) \in {\cal N}_{\textbf{X}}(\textbf{x}^*)$,   it follows that $y^* \in Y^*$ if and only if:
\[ -\langle y^*_{V^\perp}, M (\textbf{x} - \textbf{x}^*) \rangle \leq \langle \nabla g(x^*), x - x^*\rangle \quad \forall \textbf{x} \in \textbf{X},   \]
or, since $M \textbf{x}^* = M  \bar{\textbf{x}} =0$, equivalently: 
\[ -\langle y^*_{V^\perp}, M (\textbf{x} - \bar{\textbf{x}}) \rangle \leq \langle \nabla g(x^*), x - x^*\rangle \quad \forall \textbf{x} \in \textbf{X}.   \]
Let us denote the affine hull of $\cap_{i=1}^r X_i$ by  $\hat{\cal X} = \mathrm{Span}(\cap_{i=1}^r X_i)$. Now, using similar arguments as in the proof of Theorem 20.1 \cite{RocWet:98}, since
\begin{align*}
\bar x \in  \left(\bigcap_{i=1}^r \ri(X_i) \right) \bigcap \left( \bigcap_{i=r+1}^m X_i \right)
\end{align*}
and  $X_{r+1}, \cdots, X_m $ are polyhedral sets with $1 \leq r \leq m$,  we have
\begin{align*}
\bar x & \in   \left(\bigcap_{i=1}^r \ri(X_i) \right) \bigcap \ri\left( \bigcap_{i=r+1}^m (X_i \cap \hat{\cal X})\right) =    \left(\bigcap_{i=1}^r \ri(X_i) \right) \bigcap \left( \bigcap_{i=r+1}^m \ri(X_i \cap \hat{\cal X})\right).
\end{align*}
Hence, we obtain:
\begin{align*}
0  & \in   \ri \left(M_{i,:}(X_i - \bar x) \right) \quad \forall i=0:r,  \qquad 0   \in  \ri \left( M_{i,:}(X_i \cap \hat{\cal X} - \bar x) \right)  \quad \forall i=r+1:m.
\end{align*}
and consequently
\[   0 \in \ri \left( (\textbf{+}_{i=0}^r M_{i,:} (X_i - \bar x))  +  (\textbf{+}_{i=r+1}^m M_{i,:} (X_i \cap \hat{\cal X} - \bar x)) \right).  \]
Then, using the definition of the relative interior, we have that there exists $ \delta >0$  such that   for all $y^*_{V^\perp} \in V^\perp$,  we can set  $ \textbf{x} =(x; x_1; \cdots; x_m) \in \textbf{X}$ satisfying: 
\[  \sum _{i=0}^r M_{i,:} (x_i - \bar x) + \sum _{i=r+1}^m M_{i,:} (x_i - \bar x)  =   - \frac{\delta}{\|y^*_{V^\perp}\|} y^*_{V^\perp},  \]
where we used the convention $x_0 = x$.  Hence, we obtain:
\[
M(\textbf{x}  - \bar{\textbf{x}} ) = - \frac{\delta}{\|y^*_{V^\perp}\|} y^*_{V^\perp}. 
\]
Based on this relation, we further get:
\[
\delta \|y^*_{V^\perp}\| \leq \langle \nabla g(x^*), x - x^*\rangle \;.
\]
We need to ensure that $x$ can be chosen bounded.  Let us consider a slightly different  problem than optimization problem \eqref{eq:extended}:
\begin{equation}
\label{eq:aux}
\min_{\textbf{x} \in \rset^{n(m+1)}} g(x) + \sum_{i=1}^m \mathbb{I}_{X_i}(x_i) + \sum_{i=1}^m \mathbb{I}_{\{0\}}(x - x_i) + \mathbb{I}_{\{g \leq g(x^*) + 1/2\}}(x).
\end{equation}
Compared to \eqref{eq:extended} we only added a constraint on $x$ to be in a sublevel set of $g$.  From the line segment principle  \cite{RocWet:98}  problem \eqref{eq:aux}  has also  a nonempty  relative interior.  Moreover,  since $g$ is $1$-strongly convex function, this constraint enforces  $x$ to be bounded, that is  $1/2 \|x - x^*\|^2 \leq g(x) - g(x^*) \leq 1/2$, or, equivalently $ \|x - x^*\| \leq 1$. Then, since $g(x^*) < g(x^*) + 1/2$, it is clear that the primal optimal solution of \eqref{eq:aux} is the same as the one of \eqref{eq:extended}. Moreover, as the new constraint will not be active, it has no impact on the KKT conditions and thus the dual optimal sets of both problems is the same $Y^*$. Hence, $x$ can be chosen in the bounded sublevel set $\{x: g(x) \leq g(x^*)+1/2\}$ and we get
\[
 \|y^*_{V^\perp}\| \leq \frac{\|\nabla g(x^*)\|  \|x - x^*\| }{\delta} \leq \frac{\|\nabla g(x^*)\|  }{\delta},
\]
which is enough to prove the compactness of $\mathcal K$.  \hfill$\square$
\end{proof}

\noindent Now,  we are ready  to derive one of the  main results of this section, which states that  a local quadratic growth condition holds for  the dual on a ball around $Y^*$.

\begin{theorem}
\label{th_primal-dual_qg}
Assume that the sets $X_i$'s satisfy Slater's condition, that is there exists $1 \leq r \le m$ such that the sets $X_{r+1}, \cdots, X_m $ are polyhedral and there exists $\bar x \in 
\left(\bigcap_{i=1}^r \ri(X_i) \right) \bigcap \left( \bigcap_{i=r+1}^m X_i \right) $, and $g$ is 1-strongly convex and with $L$-Lipschitz continuous gradient. Then, the dual function $d$ satisfies a local quadratic growth condition,  that is 
there exists $\sigma'>0$ and $R'>0$ such that: 
\[  d(y) - d^* \geq \frac{\sigma'}{2} \dist^2(y, Y^*) \quad  \forall y: \; \dist(y, Y^*)\leq R'. \] 
\end{theorem}

\begin{proof}
From Lemma~\ref{lem:lower_bound_on_D} it follows that for any $R>0$ and any  $y$ such that $\dist(y, Y^*) \leq R$,  we have:
\[
d(y) - d^* \geq  \min\left(\frac{\sup_{u \in  {\textbf{X}} \cap \textbf{B}}-  \langle y - y^*, M u \rangle}{2}, \frac{(\sup_{u \in  {\textbf{X}} \cap \textbf{B}}-  \langle y - y^*, M u \rangle)^2}{2LD_R^2}\right)
\]
where recall that  $\textbf{B}$ denotes the ball of center $\textbf{x}^*$ and radius $D_R = \frac 2L \|M\| R$. Let us denote:
\[
\sigma = \inf_{y \in \rset^{mn} \setminus Y^*} \sup_{u \in  {\textbf{X}} \cap \textbf{B}}- \Big\langle \frac{y - \Pi_{Y^*}(y)}{\|y - \Pi_{ Y^*}(y)\|}, M u \Big\rangle = \inf_{\|d\| \leq 1} \sup_{u \in  {\textbf{X}} \cap \textbf{B}}  - \langle d, M u \rangle
\]
As $ {\textbf{X}} \cap \textbf{B}$ is compact, the function $\phi(d) =  \sup_{u \in  {\textbf{X}} \cap \textbf{B}}- \langle d, M u \rangle$ is continuous. Moreover, since the set $\{d: \|d\| \leq 1\}$  is compact, there exists $\bar y \in \rset^{mn} \setminus Y^*$ such that $\sigma = \phi(\frac{\bar y - \Pi_{Y^*}(\bar y)}{\|\bar y - \Pi_{Y^*}(\bar y)\|})$. We also conclude from  \eqref{eq:positive_for_any_y} that $\sigma > 0$. Thus, we have:
\[
d(y) - d^* \geq  \min\left(\frac{\sigma}{2}\dist(y, Y^*), \frac{\sigma^2}{2LD_R^2}\dist^2(y, Y^*)\right).
\]
Finally, let us note that $\frac{\sigma}{2}\dist(y,Y^*) \geq \frac{\sigma^2}{LD_R^2}\dist^2(y,Y^*)$ as soon as $\dist(y, Y^*) \leq \frac{L D_R^2}{\sigma}$. Hence, we get our statement  by taking $\sigma' = \frac{\sigma^2}{2LD_R^2}$ and $R' = \min(R, \frac{L D_R^2}{\sigma})$.  \hfill$\square$
\end{proof}

\noindent Finally, we show that under the above assumptions the dual satisfies a local quadratic growth condition on any sublevel set.
\begin{theorem}
\label{th_qg}
Let the assumptions of Theorem \ref{th_primal-dual_qg} hold.  Then, for any fixed  dual variable $y^0$  there exists $\sigma = \sigma(y_0)>0$ such that:
\begin{align*}
d(y) - d^* \ge \frac{\sigma}{2} \text{dist}^2(y,Y^*)^2 \quad \forall y: d(y) \le d(y^0) .
\end{align*}
\end{theorem}

\begin{proof}
From Theorem \ref{th_primal-dual_qg}  we have that there exists $ \sigma'>0$ and $R'>0$ such that 
\begin{align*}
 d(y) - d^* \ge \frac{\sigma'}{2}\text{dist}^2(y,Y^*) \qquad \forall y: \text{dist}(y,Y^*) \le R'.
\end{align*}
Take now $y$  such that $\dist(y, Y^*) > R'$ and $y^* \in Y^*$. Let us denote $y_R = y^* + \frac{R'}{\|y - y^*\|}(y-y^*)$ and observe that $\text{dist}(y_R,Y^*)\le \|y_R-y^*\| = R'$. Then,  for any $q_R \in \partial d(y_R)$ we have: 
\begin{align*}
d(y) - d^* &\geq d(y_R) + \langle q_R, y - y_R\rangle - d(y^*) \\
& = d(y_R) -d(y^*) + \frac{\|y - y^*\| - R'}{R'} \langle q_R, y_R - y^*\rangle \\
& \geq \big(1 +  \frac{\|y - y^*\| - R'}{R'}\big)(d(y_R) - d(y^*)) \\ 
&\geq  \frac{\|y - y^*\|}{R'} \cdot \frac{\sigma'}{2}\|y_R - y^*\|^2  = \frac{\sigma' R'}{2} \|y - y^*\| \geq \frac{\sigma' R'}{2}  \text{dist}(y,Y^*). 
\end{align*}
Moreover,  $\frac{\sigma'}{2}  \text{dist}^2(y,Y^*) \leq \frac{\sigma' R'}{2}  \text{dist}(y,Y^*)$ if and only if $\ \text{dist}(y,Y^*) \leq R'$. Therefore,  we get that:
\[ d(y) - d^* \geq \min \left( \frac{\sigma'}{2}  \text{dist}^2(y,Y^*), \frac{\sigma' R'}{2}  \text{dist}(y,Y^*) \right) \quad \forall y \in \rset^{mn}.\]
Now, considering $y$ only in the sublevel set   $d(y) - d^* \leq \sigma' (R')^2 / 2$, we have: 
\[ \frac{\sigma' (R')^2}{2} \geq d(y) - d^* \geq \min \left( \frac{\sigma'}{2}  \text{dist}^2(y,Y^*), \frac{\sigma' R'}{2}  \text{dist}(y,Y^*) \right),  \]
which yields then that  $y$ must satisfy  $\dist(y, Y^*) \leq R'$ and consequently $$\min \left( \frac{\sigma'}{2}  \text{dist}^2(y,Y^*), \frac{\sigma' R'}{2}  \text{dist}(y,Y^*) \right)  = \frac{\sigma'}{2}  \text{dist}^2(y,Y^*). $$ In  conclusion, we get that there exists a sublevel set where the quadratic error bound holds for the dual:
\[   d(y) - d^* \geq \frac{\sigma'}{2}  \text{dist}^2(y,Y^*)  \qquad \forall y:  d(y) \leq d^* + \sigma' (R')^2 / 2.   \]
Now, since the quadratic error bound holds for a particular sublevel set, by Proposition 1 in \cite{FerQu:18}, it follows that  the quadratic error bound holds on any given  sublevel set $\{y: \; d(y) \leq d(y^0)\}$ with a constant $\sigma$ depending on $y^0$.   \hfill$\square$
\end{proof}

\noindent Our main result of this section states that we have  on any sublevel set  a  quadratic growth condition on the dual of \eqref{bap-v}, provided that $g$ is smooth strongly convex function and the sets   $X_i$s satisfy Slater's condition. We conjecture that a similar result holds by replacing the Slater's condition with the linear regularity condition \eqref{lr}. We will investigate this conjecture in our future work.  It is important to note that when all the sets $X_i$ are polyhedral, that is $r=0$,  the statement of Theorem \ref{th_qg}  has been already proved in \cite{NecNes:15,NecNed:15}. In this paper (see Theorem \ref{th_qg}) we generalize this result to general convex sets satisfying Slater's condition (note that our result still allows that some sets to be polyhedral).   Further, we observe  that the composite form of the dual function $d(\cdot)$ in the  problem \eqref{dbap} is appropriate for the (accelerated) coordinate descent framework \cite{Nes:12,NecCli:16,LiLin:18,LuXia:14,RicTak:14}, since  $\tilde{d}(\cdot)$ is a smooth convex function, while the nonsmooth part  $\text{supp}(\cdot)$ is a separable simple convex function.  Moreover, the large  number of blocks $m$ in the dual variable $y$   represents another  motivation for using the coordinate descent approach. Therefore, in the  sequel we analyze the convergence behavior of (accelerated) coordinate descent algorithms  for solving the dual problem \eqref{dbap}, which  satisfies  a quadratic growth condition  with a constant $\sigma>0$ on a given  sublevel set $\{y: \; d(y) \leq d(y^0)\}$.


\section{Random coordinate descent}
In this section we consider a random coordinate descent algorithm  for solving the dual formulation \eqref{dbap}.   Coordinate descent algorithms and their accelerated couterparts  have been intensively studied  in the last decade thanks to their capacity to handle large-scale applications \cite{Nes:12}. Under the natural limitations of first order methods, their iteration complexity has been established for smooth  convex, strongly convex and composite  problems in e.g., \cite{FerRic:15,Nes:12,LuXia:14,RicTak:14}. Linear convergence of such algorithms under different types of relaxation of strong convexity condition (such as error bound or quadratic growth) has been derived in \cite{NecCli:16,FerQu:18}. Sublinear rate  of a primal sequence generated by  a random accelerated  dual coordinate ascent has been given recently in \cite{LiLin:18}.  For simplicity of the exposition we consider  uniform probabilities on $[m]$ for selecting the block $y_i$ of $y$.  We observe that the smooth part $\tilde{d}$ has the  (block) coordinate  gradient given by the expression: 
$$ \nabla_i \tilde{d}(y) = - \nabla g^*\left(- \sum\limits_{j=1}^m  y_j\right).  $$  
From this  it follows immediately that   the gradient of $\tilde{d}$  is block coordinate   Lipschitz   continuous  with the Lipschitz constants  $L_i = 1$ for all $i \in [m]$, since recall that we assume  $g$ to be $1$-strongly convex:
\begin{align*} \|  \nabla_i \tilde{d}(y +  U_i t_i)  - \nabla_i \tilde{d}(y) \| & = \left\|  \nabla g^*\left(- y_i - t_i - \sum\limits_{j \neq i}  y_j\right)  - \nabla g^*\left(- \sum\limits_{j=1}^m  y_j\right) \right\| \\ 
& \le  \norm{t_i}   \quad \forall t_i \in \rset^{n}, 
\end{align*}  
where $U_i$, as usual in the coordinate descent literature,  denotes  the $i$th block matrix of $I_{mn}$ corresponding to block component $y_i$ of $y$. By standard reasoning we can prove \cite{Nes:12}:
\begin{align}
\label{eq:lipJ}
 \tilde{d}(y + U_i t_i) \leq  \tilde{d}(y) + \langle \nabla_{i} \tilde{d}(y), t_i \rangle + \frac{1}{2} \norm{t_i}^2 \quad \forall t_i \in \rset^{n}.
\end{align}

\noindent It is also important to note  that for any closed convex set ${\cal X} \subseteq \rset^n$  and scalar $\alpha > 0$ the following refinement holds for the proximal operator of its support:
\begin{align}
\text{prox}_{\alpha \cdot \text{supp}_{{\cal X}}}(y) &= \arg\min_{z \in \rset^n} \left( \frac{1}{2}\norm{z - y}^2 + \alpha \cdot \text{supp}_{{\cal X}}(z) \right) \nonumber \\
& = \arg\min_{z \in \rset^n} \left( \frac{1}{2}\norm{z - y}^2 + \alpha \cdot \max_{t \in {\cal X}} \; \langle t, z \rangle \right) \nonumber \\
&= y - \alpha \cdot \arg\max_{t \in {\cal X}} \; \frac{1}{2}\norm{\alpha \cdot t - y}^2 = y - \alpha \cdot \Pi_{{\cal X}}(\alpha^{-1} \cdot y).
\label{eq:prox_supp}
\end{align}
Let us also establish a relation between the primal and  dual variables. For any dual variable $y\in \rset^{mn}$ let us define the corresponding primal variable: 
\[ x(y) =\arg \min_x \langle - \sum_{i=1}^m y_i, x \rangle - g(x). \]   Further, we observe that $$  \nabla \tilde{d} (y) = -\left( \nabla g^*(\sum_{i=1}^m y_i);\cdots;\nabla g^*(\sum_{i=1}^m y_i) \right) = - (x(y);\cdots;x(y)). $$ Similarly,  $  \nabla \tilde{d} (y^*) = -(x^*;\cdots;x^*)$ for all $y^* \in Y^*$.  Since $g$ is assumed $1$-strongly convex, then $\nabla \tilde{d}$ is $1$-Lipschitz  continuous  \cite{RocWet:98}:
\[  \|\nabla \tilde{d} (y) -  \nabla \tilde{d} (y') \|  \leq \| y - y'\| \quad \forall y, y' \in \rset^{mn}.  \]   
Now, considering $y^* = \Pi_{Y^*}(y)$ and using the quadratic growth property for the dual, we have:
\[   \|\nabla \tilde{d} (y) -  \nabla \tilde{d} (y^*) \|^2  \leq \text{dist}^2(y,Y^*) \leq \frac{2}{\sigma} (d(y) - d^*) \quad \forall y: d(y) \leq d(y^0).   \]
Using now the explicit expressions for the $\nabla \tilde{d} (y)$ and $\nabla \tilde{d} (y^*)$ derived previously, we get the following primal-dual inequality:
\begin{align}
\label{eq:connection-xy}
\|  x(y) - x^*\|^2  = \frac{1}{m}   \|\nabla \tilde{d} (y) -  \nabla \tilde{d} (y^*) \|^2  \leq \frac{2}{\sigma m} (d(y) - d^*) \quad \forall y: d(y) \leq d(y^0). 
\end{align}

\noindent Now, let $y^0 \in \rset^{mn}$ be the initial point and using that $L_i = 1$ for all $i \in [m]$, then we consider the following random  coordinate descent (RCD) scheme:
\begin{align*}
\begin{cases}
& \textbf{RCD}: \\
& \text{For} \; k\geq 0 \; \text{do}\!:\\
&\text{Choose uniformly random index} \; i_k \in [m] \; \text{and update:} \\
&y_{i_k}^{k+1} = \text{prox}_{\text{supp}_{X_{i_k}}}\left(y_{i_k}^k -   \nabla_{i_k} \tilde{d}(y^k) \right)  \\
& y_j^{k+1} = y_j^k \qquad \forall j \neq i_k.
\end{cases}
\end{align*}

\noindent Based on the particular form \eqref{eq:prox_supp} of the proximal operator of the support function of $X_{i_k}$, we can rewrite the RCD iteration in a more explicit form as:
\begin{align} 
\label{eeq:iterD}
y_{i_k}^{k+1} & = \left(y_{i_k}^k -   \nabla_{i_k} \tilde{d} (y^k) \right) -  \Pi_{X_{i_k}} \left( y_{i_k}^k -   \nabla_{i_k} \tilde{d} (y^k) \right)  \nonumber \\
& = \left(y_{i_k}^k +  \nabla g^*(- \sum_{j=1}^m  y_j^k) \right) -  \Pi_{X_{i_k}} \left( y_{i_k}^k +\nabla g^*(- \sum_{j=1}^m  y_j^k) \right). 
\end{align}

\noindent  Hence, each iteration of RCD requires a  projection onto a single simple set $X_{i_k}$.  Recall that we assume that projections onto individual sets $X_i$ are easy for all $i \in [m]$.  Additionally, at each iteration we need to also  evaluate the gradient of the Fenchel conjugate of $g$, i.e.,  $$\nabla g^*(- \sum_{j=1}^m  y_j^k)  = \arg \max_{x \in \rset^n}  \langle - \sum_{j=1}^m  y_j^k, x \rangle -g(x).  $$ Therefore, if  $\max_{x \in \rset^n}  \langle y, x \rangle -g(x)$ can be computed efficiently for any  given $y$, then we have a very fast implementation of the RCD iteration. In Section \ref{sec:dykstra} we show that the best approximation problem yields indeed an efficient implementation of RCD iteration.  Let us now analyze the convergence behavior of RCD algorithm.  It is well-known that RCD has sublinear convergence  of order ${\cal O}(1/k)$ in expectation when the smooth component has coordinate Lipschitz continuous gradient, see e.g., \cite{FerRic:15,NecCli:16,Nes:12,LuXia:14,RicTak:14}.   Moreover, linear convergence of  RCD  was proved in  \cite{FerRic:15,LuXia:14,Nes:12,RicTak:14}  for the strongly convex case and further extended in \cite{NecCli:16} to the error bound case.  Below, we provide a simple proof for the linear convergence of RCD  under the  quadratic growth condition, with better constants in the rates than e.g.  \cite{Nes:12,RicTak:14}.   
\begin{align}
\label{dual_qg}
d(y) - d^* \ge \frac{\sigma}{2} \text{dist}^2(y,Y^*)^2 \quad \forall y: d(y) \le d(y^0) .
\end{align}
Recall that, according to Theorem  \ref{th_qg}, $\sigma$ from the quadratic growth condition  \eqref{dual_qg} depends on $y^0$. 

\begin{theorem}
\label{th_convrate_dualfun}
Let the assumptions of Theorem \ref{th_primal-dual_qg} hold (hence,  the quadratic growth  \eqref{dual_qg} holds for some $\sigma >0$).   Then, the following linear convergence rate in expectation holds for the sequence $\{y^k\}_{k \ge 0}$ generated by the RCD algorithm:
\begin{align*}
\mathbb{E}[d(y^{k}) - d^*]  \le \left( 1 - \frac{\sigma}{m(\sigma + 1)}  \right)^k  \left( d(y^{0}) - d^* + \frac{1}{2}  \text{dist}^2(y^0,Y^*) \right) \quad \forall k \geq 0.
\end{align*}
\end{theorem}

\begin{proof}
By using the Lipschitz gradient property \eqref{eq:lipJ} of the dual smooth part $\tilde{d}$ we have for any $k\geq 0$:
\begin{align}
	d(y^{k+1}) &  = \tilde{d}(y^{k+1}) + \text{supp}(y^{k+1}) \nonumber\\
	& \le \tilde{d}(y^k) + \langle \nabla_{i_k} \tilde{d}(y^k), y^{k+1}_{i_k} - y^k_{i_k} \rangle + \frac{1}{2}\norm{y^{k+1}_{i_k} - y^k_{i_k}}^2 + \text{supp}(y^{k+1}) \nonumber \\
	& = \tilde{d}(y^k) + \langle \nabla_{i_k} \tilde{d}(y^k), y^{k+1}_{i_k} - y^k_{i_k} \rangle + \frac{1}{2}\norm{y^{k+1}_{i_k} - y^k_{i_k}}^2 + \text{supp}_{X_{i_k}}(y^{k+1}_{i_k}) \nonumber \\ 
	& +\sum\limits_{j \neq i_k} \text{supp}_{X_j}(y^{k}_j). \nonumber \\
	& = \min_{z_{i_k} \in \rset^{n}} \; \tilde{d}(y^k) + \langle \nabla_{i_k} \tilde{d}(y^k), z_{i_k} - y^k_{i_k} \rangle + \frac{1}{2} \norm{z_{i_k} - y^k_{i_k}}^2 + \text{supp}_{X_{i_k}}(z_{i_k}) \nonumber \\
	& +\sum\limits_{j \neq i_k} \text{supp}_{X_j}(y^{k}_j). 
	\label{composite_descent_rel}
\end{align}
Note that  the sequence $\{y^k\}_{k \ge 0}$ generated by the RCD algorithm remains  in the sublevel set given by  $y^0$, i.e. $d(y^k) \le d(y^0)$ for all $k \geq 0$, since by taking $z_{i_k} = y^k_{i_k}$ in \eqref{composite_descent_rel} we get: 
\begin{align}\label{boundedlev}
d(y^{k+1}) \le d(y^k) \qquad \forall k \geq 0.
\end{align}
Let us define the strongly convex function $z_{i_k} \mapsto \Psi(z_{i_k};y^k)$ as:
\[  \Psi(z_{i_k};y^k) \!= \tilde{d}(y^k) + \langle \nabla_{i_k} \tilde{d}(y^k), z_{i_k} \!- y^k_{i_k} \rangle + \frac{1}{2} \norm{z_{i_k} \!- y^k_{i_k}}^2 + \text{supp}_{X_{i_k}}(z_{i_k})  +\!  \sum\limits_{j \neq i_k} \! \text{supp}_{X_j}(y^{k}_j).  \]
Since $y^{k+1}_{i_k} = \arg \min_{z_{i_k} \in \rset^{n}} \Psi(z_{i_k};y^k) $ and $\Psi(z_{i_k};y^k)$ is $1$-strongly convex, we have:
\[  \Psi(z_{i_k};y^k) \geq \Psi(y^{k+1}_{i_k};y^k) +\frac{1}{2} \| z_{i_k} - y^{k+1}_{i_k}  \|^2 \quad \forall z_{i_k} \in \rset^{n}.  \]
Using this inequality in \eqref{composite_descent_rel}  we further get:
 \begin{align*}
	d(y^{k+1})  \leq & \; \tilde{d}(y^k) + \langle \nabla_{i_k} \tilde{d}(y^k), z_{i_k} - y^k_{i_k} \rangle + \frac{1}{2} \norm{z_{i_k} - y^k_{i_k}}^2 + \text{supp}_{X_{i_k}}(z_{i_k}) \\ 
	& +\sum\limits_{j \neq i_k} \text{supp}_{X_j}(y^{k}_j)  - \frac{1}{2} \| z_{i_k} - y^{k+1}_{i_k}  \|^2  \quad \forall z_{i_k} \in \rset^{n}.
\end{align*} 
By taking the conditional expectation over $i_k$ conditioned on $y^k$ on both sides of the previous relation, we obtain :	
	\begin{align*}
	& \mathbb{E}_{i_k}[d(y^{k+1}) | y^k] \le  \;    \tilde{d}(y^k) + \frac{1}{m}\left[\langle \nabla \tilde{d}(y^k), z - y^k \rangle + \frac{1}{2}\norm{z - y^k}^2 + \text{supp}(z) \right] \\ 
	& \quad + \left(1 - \frac{1}{m} \right)\text{supp}(y^k) - \frac{1}{2}   \mathbb{E}_{i_k}[ \| z_{i_k} - y^{k+1}_{i_k}  \|^2 | y^k]  \\ 
& = \left(1 - \frac{1}{m} \right) 	\tilde{d}(y^k) + \frac{1}{m}\left[ \tilde{d}(y^k) + \langle \nabla \tilde{d}(y^k), z - y^k \rangle + \frac{1}{2}\norm{z - y^k}^2 + \text{supp}(z) \right]  \\
	& \quad + \left(1 - \frac{1}{m} \right)\text{supp}(y^k) - \frac{1}{2}   \mathbb{E}_{i_k}[ \| z_{i_k} - y^{k+1}_{i_k}  \|^2 | y^k]  \\ 
& \leq 	\left(1 - \frac{1}{m} \right) d(y^k)  + \frac{1}{m}\left[  \tilde{d}(z)   + \frac{1}{2}\norm{z - y^k}^2 + \text{supp}(z) \right]  - \frac{1}{2}   \mathbb{E}_{i_k}[ \| z_{i_k} - y^{k+1}_{i_k}  \|^2 | y^k] \\
& = \left(1 - \frac{1}{m} \right) d(y^k) +  \frac{1}{m}\left[  d(z)   + \frac{1}{2}\norm{z - y^k}^2  \right]  - \frac{1}{2}   \mathbb{E}_{i_k}[ \| z_{i_k} - y^{k+1}_{i_k}  \|^2 | y^k] \\
&=  \left(1 - \frac{1}{m} \right) d(y^k) +  \frac{1}{m} d(z)  + \frac{1}{2}   \mathbb{E}_{i_k}[ \| z_{i_k} - y^{k}_{i_k}  \|^2 - \| z_{i_k} - y^{k+1}_{i_k}  \|^2 | y^k] \\ 
& =    \left(1 - \frac{1}{m} \right) d(y^k) +  \frac{1}{m} d(z)  + \frac{1}{2}   \mathbb{E}_{i_k}[  -\| y^{k}_{i_k}  - y^{k+1}_{i_k}  \|^2 + 2 \langle z_{i_k} - y^{k}_{i_k},  y^{k+1}_{i_k}  - y^{k}_{i_k} \rangle    | y^k]  
	\end{align*}
for all  $z \in \rset^{mn}$,  where in the last inequality  we used convexity of $\tilde{d}$. Choosing   $z = y^k_* :=\Pi_{Y^*} (y^k)$ in the previous inequality, we get: 
\begin{align}   
\label{eq:eik}
& \frac{1}{2}   \mathbb{E}_{i_k}[  \| y^{k}_{i_k}  - y^{k+1}_{i_k}  \|^2 + 2 \langle y^{k}_{i_k} - (y^k_*)_{i_k},  y^{k+1}_{i_k}  - y^{k}_{i_k} \rangle    | y^k]  \\
& \leq    \left(1 - \frac{1}{m} \right) d(y^k) +  \frac{1}{m} d(y^k_*)  - \mathbb{E}_{i_k}[d(y^{k+1}) | y^k].  \nonumber 
 \end{align}
On the other hand, we also have:
\begin{align*}
& \frac{1}{2}   \mathbb{E}_{i_k}[  \text{dist}^2 (y^{k+1}, Y^*) | y^k]   =  \frac{1}{2}   \mathbb{E}_{i_k}[  \| y^{k+1}  - y^{k+1}_{*}  \|^2  | y^k]  \leq   \frac{1}{2}   \mathbb{E}_{i_k}[  \| y^{k+1}  - y^{k}_{*}  \|^2  | y^k]  \\
&  =   \frac{1}{2}   \mathbb{E}_{i_k}[  \|y^k  + U_{i_k}(y^{k+1}_{i_k} -y^k_{i_k})  - y^{k}_{*}  \|^2  | y^k]  \\  
& = \frac{1}{2}   \text{dist}^2 (y^{k}, Y^*)   +  \frac{1}{2}   \mathbb{E}_{i_k}[  \| y^{k}_{i_k}  - y^{k+1}_{i_k}  \|^2 + 2 \langle y^{k}_{i_k} - (y^k_*)_{i_k},  y^{k+1}_{i_k}  - y^{k}_{i_k} \rangle    | y^k]\\
& \leq  \frac{1}{2}   \text{dist}^2 (y^{k}, Y^*)  +  \left(1 - \frac{1}{m} \right) d(y^k) +  \frac{1}{m} d(y^k_*)  - \mathbb{E}_{i_k}[d(y^{k+1}) | y^k],
\end{align*}
where in the last inequality we used \eqref{eq:eik}. Hence, taking now full expectation and subtracting $d^*$ from both sides of the previous inequality, we obtain the following recurrence:
\begin{align}
\label{eq:rec1}
& \frac{1}{2}   \mathbb{E}[ d(y^{k+1}) -d^*  +  \text{dist}^2 (y^{k+1}, Y^*)]   \\
& \leq   \mathbb{E}[ d(y^k) -d^*  + \frac{1}{2}   \text{dist}^2 (y^{k}, Y^*)]   - \frac{1}{m}   \mathbb{E}[ d(y^k) - d^*].  \nonumber 
\end{align}
Now, using the quadratic growth condition \eqref{dual_qg}, we have:
\begin{align*}   
d(y^k) - d^*  &  = \frac{\sigma}{1+\sigma} (d(y^k) - d^*) + \left(1- \frac{\sigma}{1+\sigma}  \right) (d(y^k) - d^*) \\ 
& \geq  \frac{\sigma}{1+\sigma} (d(y^k) - d^*)  + \left(1- \frac{\sigma}{1+\sigma}  \right) \frac{\sigma}{2}  \text{dist}^2 (y^{k}, Y^*)  \\
&=  \frac{\sigma}{1+\sigma} \left( d(y^k) - d^*   + \frac{1}{2}  \text{dist}^2 (y^{k}, Y^*) \right).  
\end{align*}
Using this inequality in \eqref{eq:rec1}, we further get:
\begin{align*}
\label{eq:rec1}
& \frac{1}{2}   \mathbb{E}[ d(y^{k+1}) -d^* +  \text{dist}^2 (y^{k+1}, Y^*)]  \\  
& \leq \left( 1 -  \frac{\sigma}{m(1+\sigma)} \right)  \mathbb{E}[ d(y^k) -d^*  + \frac{1}{2}   \text{dist}^2 (y^{k}, Y^*)], \nonumber 
\end{align*}
which concludes our statement.  \hfill$\square$
\end{proof}	

\noindent If we define the primal sequence given by: 
\[ x^k  =\arg \min_x \langle - \sum_{i=1}^m y_i^k, x \rangle - g(x), \]   
then from relation \eqref{eq:connection-xy} and Theorem \ref{th_convrate_dualfun}  we can also derive the following  linear rate in terms of the expected quadratic distance of the primal sequence $x^k$ to the optimal solution $x^*$:
\[   \mathbb{E} [ \| x^k - x^*\|^2 ]   \leq \frac{2}{\sigma m}    \left( 1 -  \frac{\sigma}{m(1+\sigma)} \right)^k  \left( d(y^{0}) - d^* +  \frac{1}{2}   \text{dist}^2 (y^{0}, Y^*)\right).   \]  

\begin{remark}
\label{rem:rcd}
The linear rate of convergence stated in Theorem \ref{th_convrate_dualfun} clearly implies the following estimate on the total number of iterations required by RCD to obtain an $\epsilon-$suboptimal solution in expectation:
\begin{align*}
\mathcal{O}\left( \frac{m(1+\sigma)}{ \sigma} \log\left(\frac{1}{\epsilon} \right)\right).
\end{align*}
\end{remark}


\section{ Random  accelerated coordinate descent } 
\noindent  In this section we  consider an  accelerated version of the  RCD algorithm and analyze its convergence.   Let $y^0 \in \rset^{mn}$ be the initial point and $K$ be the maximum number of iterations we want to perform. Then, we consider the following  random  accelerated coordinate descent scheme:
\begin{align*}
\begin{cases}
& \textbf{RACD}(y^0,K): \\
&\text{Set} \; \theta_0 = \frac{1}{m} \; \text{and}\; z^0 = y^0\\
& \text{For} \; k = 0: K-1 \; \text{do}:\\
&  v^k = (1 - \theta_k)y^k + \theta_k z^k \\ 
&  \text{Choose uniformly random  index} \; i_k \in [m] \; \text{and update:} \\
&  z_{i_k}^{k+1} = 
\text{prox}_{\frac{1}{\theta_k m} \cdot \text{supp}_{X_{i_k}}} \left(z^k_{i_k} -  \frac{1}{\theta_k m}\nabla_{i_k} \tilde{d}(v^k) \right)  \\
&z_j^{k+1} = z_j^k \qquad \forall j \neq i_k.\\
&  y^{k+1} = v^k + m\theta_k (z^{k+1} - z^k)\\
&  \theta_{k+1} = \frac{\sqrt{\theta_k^4 + 2 \theta_k^2} - \theta_k^2}{2}
\end{cases}
\end{align*}

\noindent It is well-known that under the quadratic growth property  accelerated gradient methods exhibit linear convergence  in combination with a restarting procedure, see e.g. \cite{NecNes:15}. Moreover, in many cases, to efficiently stop the restarted accelerated  scheme one typically needs an accurate estimate of the quadratic growth constant $\sigma$. Therefore, further we present a restarting variation of RACD as propossed in  \cite{FerQu:18}, which does not require explicit knowledge of $\sigma$.  Let $y^0 \in \rset^{mn}$ be the initial point and $K$ the maximum number of iterations, then we consider the following restarted random accelerated  coordinate descent scheme:
\begin{align*}
\begin{cases}
& \textbf{Restarted-RACD}(y^0): \\
&\text{Set} \; \tilde{y}^0 = y^0.\\
&\text{Choose restart epochs} \; \{K_0, \cdots, K_r, \cdots\}\\
& \text{For} \; r \ge 0 \; \text{do}: \\
&  \bar{y}^{r+1} = \text{ARCD}(\tilde{y}^r,K_r) \\
&  \tilde{y}^{r+1} =  \bar{y}^{r+1} \textbf{1}_{d(\bar{y}^{r+1}) \le d(\tilde{y}^{r})} + \tilde{y}^{r} \textbf{1}_{d(\bar{y}^{r+1}) > d(\tilde{y}^{r})}  \\
\end{cases}
\end{align*}

\noindent Note that the update rule for $ \tilde{y}^{r+1}$ forces this restarted iterative process to produce sequences of points at the end of each epoch that are always in the sublevel set given by  $ \tilde{y}^{r}$ of  the previous epoch, and consequently in the original sublevel set  given by $y^0$, where our dual function satisfies the quadratic growth condition \eqref{dual_qg}.   Further, we briefly present the complexity estimate for the Restarted-RACD algorithm, more  details can be found  in \cite{FerQu:18}. We will use the general index notation $ z^{j,p} $ as the $j$th iterate from $p$th epoch and we define the following constants:
\[   \beta = \left\lceil \max\left(0, \log_2(K^*/K_0) \right)\right\rceil \quad  \text{and}  \quad K^* = \left\lceil \frac{2e}{\theta_0} \left(\sqrt{\frac{1 + \sigma}{\sigma}} - 1 \right) + 1\right \rceil. \]
We also define the iteration sequence: 
\[  y^k = \begin{cases}\tilde{y}^p & \text{if} \; j_k = K_{p-1} \\ y^{j_k,p} & \text{otherwise} \end{cases}, \;\; \text{where}  \; j_k  = k -  \sum\limits_{i = 1}^{p} K_i. \]
Finally, we fix the length of the first epoch, $K_0$,  to:
\[  K_0 =  \left\lceil \frac{2e}{\theta_0} \left(\sqrt{\frac{1 + \bar \sigma}{\bar \sigma}} - 1 \right) + 1\right \rceil,  \]
where $\bar{\sigma}$ is an estimate of the unknown constant $\sigma$ from \eqref{dual_qg}. Then, we have the following linear convergence result:  

\begin{theorem}
\label{th_convrate_dualfun-a}
Let the assumptions of Theorem \ref{th_primal-dual_qg} hold and  the sequence $\{K_j\}_{j \ge 0} \subset \mathbb{N}$ satisfy: $(i) K_{2^j - 1} = 2^j K_0$ for all $j \in \mathbb{N}$; 
$(ii) |\{ 0 \le r < 2^p - 1 \; | \; K_r = 2^j K_0 \} | =  2^{p-1-j}$ for all $j \in [p]$. Then, after $p$ epochs we have the following linear rate in expectation:
\begin{align*}
\mathbb{E}\left[d(y^{k})-  d^*\right] 
& \le \left( e^{- \frac{4}{(p + 2)2^{\beta} K_0 }} \right)^{k} (d(y^{0}) - d^*).
\end{align*}
\end{theorem}

\begin{proof}
	Let us define the constant:   
	\begin{align*}
	c_{\beta}(p) = |\left\{ l < 2^p - 1 \; | \; K_l\ge 2^{\beta} K_0  \right\} | + 1 = 1 + \sum\limits_{k = \beta}^{p - 1} 2^{p - 1 - k} =  2^{p- \beta}.
	\end{align*}
Recall also that:
	\begin{align*}
	{  z^{j,p} }  = \begin{cases} \tilde{y}^p & \text{if} \; j = K_{p-1} \\ y^{j,p} & \text{otherwise}\end{cases}. 
	\end{align*}
	Thus, we have: $z^{K_{p-1},p} = z^{0,p+1}$.
	Note that $c_\beta(p)$ represents the number of epochs such that $K_l \ge K^*$. On the other hand, we have:
	\begin{align}
	\label{linrate1}
	\mathbb{E}\left[d(y^{2^p - 1})-  d^*\right]  &= \mathbb{E}\left[d(z^{0,2^p})-  d^*\right] \le e^{-2 c_{\beta}(p)} (d(z^{0,0}) - d^*) \\
	&= e^{-2 c_{\beta}(p)} (d(y^{0}) - d^*). \nonumber 
	\end{align}
Then, we get:
	\begin{align*}
	k = \sum\limits_{i = 0}^{2^p - 1} K_i 
	& = \sum\limits_{j = 0}^{p} \left| \left\{0 \le r <2^p -1 \; | \; K_r = 2^jK_0\right\} \right| \cdot 2^jK_0 + K_{2^p - 1} \\
	& = \sum\limits_{j = 0}^{p} 2^{p-1 - j}2^jK_0 + K_{2^p - 1} = (p + 2) 2^{p-1}K_0,
	\end{align*}
	and the relation \eqref{linrate1} implies:
	\begin{align*}
	\mathbb{E}\left[d(y^{k})-  d^*\right] 
	& \le \left( e^{-2 \frac{c_{\beta}(p)}{\sum\limits_{i=0}^{2^{p}-1} K_i }} \right)^{k} (d(y^{0}) - d^*)  = \left( e^{- \frac{2^{p-\beta + 1}}{(p + 2)2^{p-1} K_0 }} \right)^{k} (d(y^{0}) - d^*),
	\end{align*}
	which confirms the above result.	\hfill$\square$
\end{proof}

\noindent If we define a primal sequence, as in the RCD case,  given by: 
\[ x^k  =\arg \min_x \langle - \sum_{i=1}^m y_i^k, x \rangle - g(x), \]   
then from relation \eqref{eq:connection-xy} and Theorem \ref{th_convrate_dualfun-a}  we can also derive a linear rate in terms of the expected quadratic distance of this primal sequence $x^k$ to the optimal solution $x^*$:
\[   \mathbb{E} [ \| x^k - x^*\|^2 ]   \leq \frac{2}{\sigma m}    \left( e^{- \frac{4}{(p + 2)2^{\beta} K_0 }} \right)^{k}   \left( d(y^{0}) - d^* \right).   \]

\begin{remark}
\label{rem:racd}
From Theorem \ref{th_convrate_dualfun-a}  it follows that  an upper bound on the total number of iterations performed by the Restarted-RACD	scheme to attain an $\epsilon-$suboptimal solution in expectation is given by:
$$ \mathcal{O}\left(\frac{m}{ \sqrt{\sigma}}\log \left( \frac{1}{\epsilon} \right)  \log_2 \left( \log \left( \frac{ d(y^0) - d^*}{\epsilon} \right) \sqrt{\frac{\bar{\sigma}}{\sigma}} \right) \right),$$
where recall that $\bar{\sigma}$ is an estimate of the unknown constant  $\sigma$. If we assume for simplicity that $d(y^0) - d^* \leq 1$ and $\sigma$ is known, then the previous estimate corresponding to Restarted-RACD is better than the estimate from Remark \ref{rem:rcd}  corresponding to RCD, provided that $\epsilon$ is sufficiently large. More precisely, the desired accuracy and sigma must satisfy $  \log_2 (\log(1/\epsilon))  \leq \sqrt{\sigma} + 1/\sqrt{\sigma}$.
For instance, if $\sigma = 10^{-2}$, then Restarted-RACD has a better worst case complexity that RCD for all accuracies  $\epsilon > 10^{-477}$.
\end{remark}


\section{ Dykstra type algorithms}
\label{sec:dykstra}
Let us now consider the application of the results from the previous sections to  the best approximation problem i.e.  finding  the best approximation  to a given  point   $v \in \rset^n$ from  the intersection of some closed convex sets $\cap_{i=1}^m  X_i$. For convenience, we recall this problem  here: 
\begin{align}
\label{bap-d}
\min_{x \in \rset^n} & \;  \frac{1}{2}\norm{x -v}^2    \quad   \text{s.t.}  \;\;\; x \in \bigcap_{i=1}^m  X_i.
\end{align} 
Note that for this particular problem the objective function  $g(x) = \frac{1}{2}\norm{x -v}^2$ is $1$-strongly convex and with $1$-Lipschitz continouos gradient.  Moreover, the optimal solution of the best approximation  problem  is $x^* = \Pi_{\cap_{i=1}^m X_i}(v)$.  Given the particular structure of \eqref{bap-d} we can derive  a  tighter  relation between the primal and dual variables $(x(y), y)$ than in  \eqref{eq:connection-xy}. 

\begin{theorem}
\label{th:pd}
For the best approximation problem, where the sets  $\{X_i\}_{i=1}^m$ satisfy Slater's condition,   the following relation holds:
\[   \frac{1}{2} \| x(y)  - \Pi_{\cap_{i=1}^m X_i}(v)   \|^2 \leq d(y) - d^* \qquad \forall y \in \rset^{mn}.   \]
\end{theorem}

\begin{proof}
Note that for the best approximation problem, since the sets $\{X_i\}_{i=1}^m$ satisfy Slater's condition, then  there is no duality gap and  the relation between  the primal and dual variables is given by:  
\[ x(y) = v  - \sum\limits_{j=1}^m y_j . \]
Hence, we can write the dual function explicitly in terms of  $x(y)$ as follows:
\[   d(y) = \frac{1}{2} \| x(y)\|^2  -  \frac{1}{2} \|v\|^2  + \sum_{i=1}^m \text{supp}_{X_i} (y_i).  \]
Similarly, the optimal value $d^*$ can be written in terms of some optimal dual variable  $y^* \in Y^*$ as:
\[   d^* = d(y^*) = \frac{1}{2} \| \Pi_{\cap_{i=1}^m X_i}(v)  \|^2  - \frac{1}{2}  \|v\|^2  + \sum_{i=1}^m \text{supp}_{X_i} (y_i^*).  \] 
Using these relations, we further have:
\begin{align*}
& \frac{1}{2} \| x(y)  - \Pi_{\cap_{i=1}^m X_i}(v)   \|^2  =   \frac{1}{2} \| x(y)\|^2   + \frac{1}{2} \| \Pi_{\cap_{i=1}^m X_i}(v)  \|^2 - \langle x(y), \Pi_{\cap_{i=1}^m X_i}(v) \rangle \\
& = d(y) - d^*  +  \| \Pi_{\cap_{i=1}^m X_i}(v)  \|^2 - \langle x(y), \Pi_{\cap_{i=1}^m X_i}(v) \rangle + \sum_{i=1}^m \text{supp}_{X_i} (y_i^*) - \text{supp}_{X_i} (y_i). 
\end{align*}
However,  from the optimality conditions of the dual problem we have $\text{supp}_{X_i} (y_i^*) = \langle  \Pi_{\cap_{i=1}^m X_i}(v), y_i^*\rangle $  and $v  - \sum\limits_{j=1}^m y_j^* = \Pi_{\cap_{i=1}^m X_i}(v)$. Using these relations, we further get:
\begin{align*}
& \frac{1}{2} \| x(y)  - \Pi_{\cap_{i=1}^m X_i}(v)   \|^2  =   \\
& = d(y) - d^*  +  \| \Pi_{\cap_{i=1}^m X_i}(v)  \|^2 - \langle x(y), \Pi_{\cap_{i=1}^m X_i}(v) \rangle + 
\langle  \Pi_{\cap_{i=1}^m X_i}(v),  \sum_{i=1}^m y_i^*\rangle \\ 
&  \qquad -  \sum_{i=1}^m \text{supp}_{X_i} (y_i)\\
& =  d(y) - d^*  +  \| \Pi_{\cap_{i=1}^m X_i}(v)  \|^2 - \langle x(y), \Pi_{\cap_{i=1}^m X_i}(v) \rangle +  \langle  \Pi_{\cap_{i=1}^m X_i}(v),  v - \Pi_{\cap_{i=1}^m X_i}(v) \rangle \\ 
& \qquad  -  \sum_{i=1}^m \text{supp}_{X_i} (y_i)\\ 
& = d(y) - d^*  + \langle v -  x(y), \Pi_{\cap_{i=1}^m X_i}(v) \rangle  -  \sum_{i=1}^m \text{supp}_{X_i} (y_i)\\
& = d(y) - d^*  + \langle \sum\limits_{j=1}^m y_j, \Pi_{\cap_{i=1}^m X_i}(v) \rangle  -  \sum_{i=1}^m \text{supp}_{X_i} (y_i) \\
& = d(y) - d^*   +   \sum_{i=1}^m  \left( \langle y_i, \Pi_{\cap_{i=1}^m X_i}(v) \rangle   - \text{supp}_{X_i} (y_i) \right) \leq d(y) - d^*,
\end{align*}  
where in the last inequality we used that $\text{supp}_{X_i} (y_i) \geq \langle y_i, \Pi_{\cap_{i=1}^m X_i}(v) \rangle $. Indeed, this follows from the definition of the support function of the set $X_i$ and the fact that  $\Pi_{\cap_{i=1}^m X_i}(v)  \in X_i$, i.e.:
\[  \text{supp}_{X_i} (y_i) = \max_{x_i \in X_i} \langle y_i, x_i \rangle \geq  \langle y_i, \Pi_{\cap_{i=1}^m X_i}(v) \rangle. \]
This concludes our proof.   \hfill$\square$
\end{proof}
 
\noindent One of the first projection-based schemes for finding the projection of a point  into an  intersection, i.e. for solving the best approximation  problem \eqref{bap-d}, is the Dykstra algorithm \cite{BoyDyk:86,ComPes:11}. The initial variant proposed in  \cite{BoyDyk:86} performs projections in a cyclic fashion:  

\begin{align*}
\begin{cases}
& \textbf{Dykstra}:\\
& \text{Set} \; x^0 = v, \;\; y^{-(m-1)} = y^{-(m-2)} = \cdots = y^0 = 0 \\
& \text{For} \; k \ge 0 \; \text{do}: \\
&  x^{k+1}  = \Pi_{X_{((k+1) \; \text{mod} \;  m)}}(x^k + y^{k+1-m}) \\ 
& y^{k+1} = y^{k+1-m} +  x^k - x^{k+1}
\end{cases}
\end{align*}

\noindent  It has been shown that the sequence $\{x^k\}_{k \ge 0}$ generated by Dysktra algorithm  convergence linearly  towards the projection $\Pi_{\cap_{i=1}^m X_i}(v)$, provided that the sets $X_i$ are all polyehdral, see e.g., \cite{DeuHun:94,Pan:17}. It is natural to ask whether such a scheme has a similar linear convergence behavior for more general sets. To the best of our knowledge, it has not been answered yet to this question. In this section we   answer positively, proving that a random variant of Dykstra algorithm converges linearly when the collection of the sets $\{X_i\}_{i=1}^m$  satisfies  the Slater's condition, i.e. there exists $1 \leq r \le m$ such that the sets $X_{r+1}, \cdots, X_m $ are polyhedral and there exists $\bar x \in  \left(\bigcap_{i=1}^r \ri(X_i) \right) \bigcap \left( \bigcap_{i=r+1}^m X_i \right) $.  In order to prove this, it is important to recognize that RCD algorithm applied directly on the dual of  the best approximation problem \eqref{bap-d} leads to  a randomized variant of Dykstra algorithm.   Indeed, let $y^0 =0$ and $x^0  = v$.  Recall that for the best approximation problem  the relation between  the primal and dual variables is given by $x(y) = v  - \sum_{j=1}^m y_j $. Furthermore, in this particular case $\nabla_i \tilde{d}(y) = - \nabla g^* (- \sum_{j=1}^m  y_j ) = \sum_{j=1}^m  y_j -v$. Then,  the RCD iteration can be written explicitly  in terms of projections  of a primal-dual sequence $(x^k,y^k)$  as:
\[   x^k =  x(y^k) = v - \sum\limits_{j=1}^m y_j^k \quad \text{and} \quad  y_{i_k}^{k+1} = y_{i_k}^{k} +  x^k -  \Pi_{X_{i_k}} (y_{i_k}^{k} +  x^k).  \]
From these relations we can easily notice that:
\[  x^{k+1} = v - \sum\limits_{j=1}^m y_j^{k+1} =  v - \sum\limits_{j=1}^m y_j^{k} + y_{i_k}^{k} - y_{i_k}^{k+1} =  x^k + y_{i_k}^{k} - y_{i_k}^{k+1} =  \Pi_{X_{i_k}}  \left(x^k + y_{i_k}^k \right).  \]
Hence,  RCD algorithm  becomes  Randomized Dykstra, which updates the  primal-dual sequences $\{x^k,y^k\}_{k\geq 0}$ as follows: 

\begin{align*}
\begin{cases}
&  \textbf{Random Dykstra}:\\
&  \text{Set} \; x^0 = v, \;\; y^0 = 0. \; \text{For} \;  k \geq 0 \; \text{do}: \\
&\text{Choose uniformly a random index} \; i_k \in [m] \; \text{and update:} \\
&x^{k+1} =  \Pi_{X_{i_k}}  \left(x^k + y_{i_k}^k \right) \\
&y_{i_k}^{k+1} = y_{i_k}^k + x^k -x^{k+1}, \;\; y_{j}^{k+1} = y_{j}^k \;\;\;  \forall j \not = i_k.
\end{cases}
\end{align*}

\noindent Note that  the Random Dykstra  algorithm requires at each iteration one  projection onto a single set from the intersection  and few vector operations in $\rset^n$.  Hence, it can be  efficiently implemented in practice, provided that each set from the intersection is simple, which recall it is one of our basic assumptions.  Moreover,  for the best approximation problem $d(0) = 0$ and thus $d(y^0) - d^* = d(0) - d^* = - d^* = g^* =  1/2 \|  v -  \Pi_{\cap_{i=1}^m X_i}(v) \|^2$. Then, combining the result of Theorem \ref{th:pd}  with the convergence rate of RCD from Theorem \ref{th_convrate_dualfun}, we get immediately the following convergence rate for the primal iterates of the Randomized Dykstra algorithm.

\begin{corollary}
If the collection of sets $\{X_i\}_{i=1}^m$ of the best approximation problem  \eqref{bap-d} satisfy Slater's condition, then there exists some constant $\sigma>0$ such that the primal sequence  $\{x^k\}_{k\geq 0}$ of  the Random Dykstra algorithm has the following linear  convergence rate in expectation:
\begin{align*}
& \mathbb{E}[ \| x^{k}  -  \Pi_{\cap_{i=1}^m X_i}(v) \|^2 ] \le  \left( 1 - \frac{\sigma}{m (\sigma+1)}  \right)^k  \left( \|  v -  \Pi_{\cap_{i=1}^m X_i}(v) \|^2 + \text{dist}^2(0,Y^*)   \right).
\end{align*} 
\end{corollary}

\begin{remark}
Note that the existing convergence results for Dykstra algorithm usually  require $y^0 =0$ \cite{BoyDyk:86,ComPes:11,DeuHun:94}. On the other hand,  our convergence analysis works for a general initialization $y^0$. For example, if we have available some $\bar{y}$ such that $d(\bar{y}) + 1/2 \text{dist}^2(\bar{y},Y^*)  < 1/2 \text{dist}^2(0,Y^*) $, then we should initialize Randomized Dykstra with this point, i.e. $y^0=\bar{y}$ instead of $y^0 =0$.   
\end{remark}

\noindent  In the accelerated case let us  define the following  primal sequences:  
\[ x^k = x(y^k) = v - \sum_{j=1}^m y_j^k,  \quad  \hat x^k = x(v^k) = v - \sum_{j=1}^m v_j^k,  \quad \tilde x^k = x(z^k) = v - \sum_{j=1}^m z_j^k. \]  
Then,  we obtain the following  primal-dual updates: 
\begin{align*}
\hat x^k  &= v - \sum_{j=1}^m v_j^k = v - \sum_{j=1}^m \Big((1 - \theta_k) y^k_j + \theta_k z^k_j\Big)  \\ 
& =  (1 - \theta_k) \Big(v - \sum_{j=1}^m y^k_j\Big)  + \theta_k \Big(v -\sum_{j=1}^m  z^k_j \Big) = (1-\theta_k) x^k + \theta_k \tilde x^k,  \\
z_{i_k}^{k+1} & = \text{prox}_{ \frac{1}{\theta_k m} \cdot \text{supp}_{X_{i_k}}} \left(z^k_{i_k} + \frac{1}{\theta_k m} x(v^k) \right) =  \text{prox}_{ \frac{1}{\theta_k m}  \cdot \text{supp}_{X_{i_k}}} \left(z^k_{i_k} + \frac{1}{\theta_k m} \hat x^k  \right)\\
&= z^k_{i_k} +  \frac{1}{\theta_k m} \hat x^k   -
 \frac{1}{\theta_k m} \Pi_{{X_{i_k}}}\left(\theta_k m z^k_{i_k} + \hat x^k\right), \\
x^{k+1} &= v - \sum_{j=1}^m y_j^{k+1} = v - \sum_{j=1}^m \Big(v^k_j + m \theta_k (z^{k+1}_j - z^k_j)\Big) \\ & = \left(v - \sum_{j=1}^m v^k_j \right) -  m \theta_k  \left(  z^{k+1}_{i_k}  - z^k_{i_k} \right)  = \Pi_{{X_{i_k}}}\left(   \hat x^k + \theta_k m z^k_{i_k} \right),\\
\tilde x^{k+1} & =  v  - \sum_{j=1}^m z_j^{k+1}=  v  - \sum_{j=1}^m z_j^{k} - z_{i_k}^{k+1} + z_{i_k}^{k}  =  \tilde x^k - z_{i_k}^{k+1} + z_{i_k}^{k} \\ 
& =  \tilde x^k + \frac{1}{\theta_k m}  \left( x^{k+1}  -  \hat x^{k} \right).  
\end{align*}
Thus, on each epoch we apply the following Random Accelerated  Dykstra  algorithm, which updates  the primal-dual sequences $\{x^k, \hat x^k, \tilde x^k,z^k\}_{k \geq 0}$ as follows: 

\begin{align*}
\begin{cases}
& \textbf{Random Accelerated  Dykstra}(y^0,K):\\
&\text{Set} \; \theta_0 = \frac{1}{m} \; \text{and}\; y^0 = z^0 \; \text{and}\; x^0 = \tilde x^0 = v - \sum_{j=1}^m z_j^0.\\
& \text{For} \; k \in \{0, \cdots, K-1 \} \; \text{do:}\\
& \text{Choose uniformly random} \; i_k \in [m] \; \text{and update} \\
& \hat x^k = (1-\theta_k) x^k + \theta_k \tilde x^k \\
& x^{k+1} = \Pi_{{X_{i_k}}}\left(   \hat x^k + \theta_k m z^k_{i_k} \right) \\
& \tilde x^{k+1} = \tilde x^k + \frac{1}{\theta_k m}  \left( x^{k+1}  -  \hat x^{k} \right) \\
& z_{i_k}^{k+1} = z^k_{i_k} +  \frac{1}{\theta_k m}  \left( \hat x^k  - x^{k+1} \right), \;\;  z_j^{k+1} = z_j^k \;\; \forall j \neq i_k \\
& \theta_{k+1} = \frac{\sqrt{\theta_k^4 + 2 \theta_k^2} - \theta_k^2}{2}.
\end{cases}
\end{align*}

\noindent Note that $\hat x^k$ can be eliminated from the Random Accelerated  Dykstra algorithm and update only two primal sequences $\{x^k,  \tilde x^k\}_{k \geq 0}$ and one dual sequence $\{z^k\}_{k \geq 0}$.  We keep the above formulation to show the similarities between the accelerated Dykstra scheme and its non-accelerated counterpart.   Moreover,  compared to RACD, the new Random Accelerated  Dykstra  algorithm has a smaller memory footprint.  Furthermore, it requires at each iteration one single projection  and few vector operations in $\rset^n$. Hence, the  computational effort per  iteration for the Random Accelerated  Dykstra is comparable  to the  Random   Dykstra algorithm.   Moreover, from previous derivations, since the Restarted  Random Accelerated  Dykstra  is equivalent to the Restarted-RACD scheme,  it is obvious that we maintain the rate of convergence from Theorem \ref{th_convrate_dualfun-a}.  More precisely, assuming that we initialize  the Restarted  Random Accelerated  Dykstra  in the first epoch with  $y^0=0$, then  combining Theorems \ref{th_convrate_dualfun-a} and  \ref{th:pd},  we get:  

\begin{corollary}
\label{cor:arcd}
If the collection of sets $\{X_i\}_{i=1}^m$ of the best approximation problem  \eqref{bap-d} satisfy Slater's condition, then  the primal sequence $\{x^k\}_{k\geq 0}$ of  the Restarted Random Accelerated  Dykstra algorithm after $p$ epochs $\{K_0,\cdots,K_{p-1}\}$  has the following linear  convergence rate in expectation:
\begin{align*}
& \mathbb{E}[ \| x^{k}  -  \Pi_{\cap_{i=1}^m X_i}(v) \|^2 ] \le  \left( e^{- \frac{4}{(p + 2)2^{\beta} K_0 }} \right)^{k} \| v  -  \Pi_{\cap_{i=1}^m X_i}(v) \|^2,  
\end{align*} 
\end{corollary} 
However, according to Remark \ref{rem:racd},  the  Random Accelerated  Dykstra with restart over epochs of length $\{K_0,\cdots,K_{r}, \cdots\}$, with $r \geq 0$ and $K_r$ taken as in Theorem \ref{th_convrate_dualfun-a}, the Random Accelerated  Dykstra   will usually lead to a faster convergence rate than the Random   Dykstra algorithm.




\end{document}